\documentclass[11pt,oneside,reqno]{article}
\usepackage[margin=1.2in]{geometry}

\usepackage{amsmath,amsthm,amssymb,color}
\usepackage[authoryear]{natbib}
\usepackage{booktabs}

\RequirePackage{amsthm,amsmath,amsfonts,amssymb,mathrsfs,dsfont,mathtools,thmtools}
\RequirePackage{natbib}
\RequirePackage[colorlinks,citecolor=blue,urlcolor=blue]{hyperref}
\usepackage[capitalize]{cleveref}
\RequirePackage{graphicx}

\crefname{equation}{}{}
\crefformat{equation}{\textup{(#2#1#3)}}
\crefrangeformat{equation}{\textup{(#3#1#4)--(#5#2#6)}}
\crefdefaultlabelformat{#2\textup{#1}#3}
\crefname{lemma}{Lemma}{Lemmas}
\crefname{page}{p.}{pp.}

\usepackage{bbm}
\usepackage{mathrsfs}
\usepackage{graphicx}
\usepackage{tikz}
\usepackage{enumitem}
\usetikzlibrary{arrows, automata}
\usetikzlibrary{calc}
\usetikzlibrary{positioning}

\numberwithin{equation}{section}
\allowdisplaybreaks[4]

\theoremstyle{plain}
\newtheorem{theorem}{Theorem}[section]
\newtheorem{proposition}{Proposition}[section]

\newtheorem{definition}{Definition}[section]

\theoremstyle{definition}

\makeatletter

\newcount\minute
\newcount\hour
\newcount\hourMins
\def\now{%
\minute=\time%
\hour=\time \divide \hour by 60%
\hourMins=\hour \multiply\hourMins by 60%
\advance\minute by -\hourMins%
\zeroPadTwo{\the\hour}:\zeroPadTwo{\the\minute}%
}
\def\zeroPadTwo#1{\ifnum #1<10 0\fi#1}

\renewcommand{\cite}{\citet}

\def\^#1{\ifmmode {\mathaccent"705E #1} \else {\accent94 #1} \fi}
\def\~#1{\ifmmode {\mathaccent"707E #1} \else {\accent"7E #1} \fi}

\def\*#1{#1^\ast}
\edef\-#1{\noexpand\ifmmode {\noexpand\bar{#1}} \noexpand\else \-#1\noexpand\fi}
\def\>#1{\vec{#1}}
\def\.#1{\dot{#1}}

\def\wt#1{\widetilde{#1}}

\def\atop{\@@atop}
\def\*#1{\mathscr{#1}}

\renewcommand{\leq}{\leqslant}
\renewcommand{\geq}{\geqslant}

\newcommand{\eq}{\eqref}

\newcommand{\IE}{\mathbbm{E}}

\newcommand{\IR}{\mathbb{R}}

\def\be#1{\begin{equation*}#1\end{equation*}}
\def\ben#1{\begin{equation}#1\end{equation}}

\def\beqn#1\eeqn{\begin{align}#1\end{align}}
\def\beq#1\eeq{\begin{align*}#1\end{align*}}

\usepackage{graphicx}
\usepackage{latexsym}
\usepackage{amsmath,amsthm,amssymb,amscd}
\usepackage{epsf,amsmath}

\def\E{{\IE}}

\newcommand{\mcl}[1]{\mathcal{#1}}

\def\blfootnote{\xdef\@thefnmark{}\@footnotetext}

\makeatother

\begin{document}

\title{
Second-order Approximation of Exponential Random Graph Models
}
\author{Wen-Yi Ding and Xiao Fang}
\date{\it Soochow University and The Chinese University of Hong Kong} 
\maketitle

\noindent{\bf Abstract:} 
Exponential random graph models (ERGMs) are flexible probability models allowing edge dependency. However, it is known that, to a first-order approximation, many ERGMs behave like Erd\"os--R\'enyi random graphs, where edges are independent. In this paper, to distinguish ERGMs from Erd\"os--R\'enyi random graphs, we consider second-order approximations of ERGMs using two-stars and triangles. We prove that the second-order approximation indeed achieves second-order accuracy in the triangle-free case. The new approximation is formally obtained by Hoeffding decomposition and rigorously justified using Stein's method.

\medskip



\section{Introduction}

\subsection{Background}

Exponential random graph models (ERGMs) are frequently used as parametric statistical models in network analysis; see, for example, \cite{wasserman1994social}. We refer to \cite{bhamidi2011mixing} and \cite{chatterjee2013estimating} for the history and the following formulation of ERGM. Let $\mcl{G}_n$ be the space of all simple graphs\footnote{In this paper, simple graphs mean undirected graphs without self-loops or multiple edges} on $n$ labeled vertices. ERGM assigns probability
\begin{equation}\label{ERGM}
p_{\beta}(G)=\frac{1}{Z(\beta)}\exp\left\{ n^2 \sum_{i=1}^k \beta_i t(H_i, G)\right\}
\end{equation}
to each $G\in \mcl{G}_n$, where $\beta=(\beta_1,\dots, \beta_k)$ is a vector of real parameters, $H_1,\dots, H_k$ are (typically small) simple graphs without isolated vertices, $|\text{hom}(H,G)|$ denotes the number of homomorphisms of $H$ into $G$ (i.e., edge-preserving injective maps $V(H)\to V(G)$, where $V(H)$ and $V(G)$ are the vertex sets of $H$ and $G$ respectively), $|\cdot|$ denotes cardinality when applied to a set,
\be{
t(H, G):=\frac{|\text{hom}(H,G)|}{n^{|V(H)|}}
}
denotes the homomorphism density, and $Z(\beta)$ is a normalizing constant. The $n^2$ and $n^{|V(H)|}$ factors in \eq{ERGM} ensure a nontrivial large $n$ limit. In this paper, we always take $H_1$ to be an edge and assume $\beta_2,\dots, \beta_k$ are positive ($\beta_1$ can be negative), although our formal approximation also applies to negative $\beta$'s. 
Note that if $k=1$, then \eq{ERGM} is the Erd\"os--R\'enyi model $G(n,p)$ where every edge is present with probability $p=e^{2\beta_1}(1+e^{2\beta_1})^{-1}$, independent of each other. If $k\geq 2$, \eq{ERGM} "encourages" the presence of the corresponding subgraphs. 
For example, if $k=2$ and $H_2$ is a rectangle, then \eq{ERGM} becomes
\begin{equation}\label{rect}
p_{\beta_1, \beta_2}(G)=\frac{1}{Z(\beta_1, \beta_2)} \exp\left\{2\beta_1 E +\frac{8\beta_2}{n^2}\square \right\},
\end{equation}
where $E$ and $\square$ denote the number of edges and rectangles in the graph $G$, respectively, and the constants 2 and 8 are the number of automorphisms of an edge and a rectangle, respectively.

Because of the non-linear nature of \eq{ERGM}, ERGMs are notoriously more difficult to analyze than classical exponential families of distributions. The ground-breaking works by \cite{bhamidi2011mixing} and \cite{chatterjee2013estimating} reveal the following fact: In a certain parameter region (called the \emph{subcritical} region), \eq{ERGM} behaves like an Erd\"os--R\'enyi random graph $G(n, p)$ in a suitable sense. Here, $0<p:=p(\beta)<1$ is determined by the parameters $\beta$ in \eq{ERGM}. In fact, the subcritical region can fill the whole space except for a set of parameters with Lebesgue measure zero; see, for example, the two-star ERGM studied in \cite{mukherjee2023statistics}. Moreover, ERGMs with different parameters can lead to the same $p$; see \cite[Fig.~2]{chatterjee2013estimating} for the case of the triangle ERGM.

\subsection{Second-order expansion}

Because ERGMs may be indistinguishable from $G(n, p)$ to a first-order approximation, we consider the second-order approximation in this paper. For this purpose, we first explain a new way of obtaining $p$ in the first-order approximation using Hoeffding decomposition (\cite{hoeffding1948}). 
Taking \eq{rect} as an example, we rewrite it as
\begin{equation}\label{rect1}
p_{\beta_1, \beta_2}(G)=\frac{1}{Z(\beta_1, \beta_2)} \exp\left\{\Big[2\beta_1+8\beta_2 p^3\Big] E +\frac{8\beta_2}{n^2}\Big[\square-n^2 p^3 E \Big] \right\}.
\end{equation}
Suppose $p$ satisfies
\begin{equation}\label{rect2}
2\Phi(p):=2\beta_1+8\beta_2p^3= \log \frac{p}{1-p}.
\end{equation}
Then, up to a constant factor independent of $G$, the probability of $G$ is proportional to
\begin{equation}\label{rect3}
p_{\beta_1, \beta_2}(G)\propto  \exp\left\{\frac{8\beta_2}{n^2}\Big[\square-n^2 p^3 E \Big] \right\} p^E (1-p)^{n-E}.
\end{equation}
What appears inside the brackets $[\dots]$ in \eq{rect3} is the number of rectangles $\square$ subtracted by its approximate leading term in the Hoeffding decomposition\footnote{We remark that the term $n^2p^3 E$ should be $(n-2)(n-3)p^3 E$ in the Hoeffding decomposition. However, the difference is negligible for our purpose. This remark applies to all the approximate Hoeffding decompositions considered in this paper.} under the Erd\"os--R\'enyi model $G(n,p)$. 
Therefore, in a suitable sense (cf. \eq{rect5}), we can hope the ERGM in \eq{rect} to be close to $G(n,p)$
whose probability distribution is
\ben{\label{p1rect}
p_1(G):=p_{1,\beta_1,\beta_2}(G)\propto p^E (1-p)^{n-E}
}
for $p$ satisfying \eq{rect2}. The equation \eq{rect2} is exactly what appears in \cite{bhamidi2011mixing} and \cite{chatterjee2013estimating}, although they used different methods to derive it, namely, Glauber dynamics (\cite{glauber1963time}) and non-linear large deviations (\cite{chatterjee2011large}), respectively. The subcritical region in this case contains parameters $\beta_1, \beta_2$ such that there is a unique solution $p\in (0,1)$ to \eq{rect2}
and $\varphi'(p)<1$, where
\begin{equation}\label{varph}
\varphi(a):=\frac{e^{2\Phi(a)}}{e^{2\Phi(a)}+1}, \quad \Phi(a):=\beta_1+4\beta_2 a^3.    
\end{equation}

Next, we use the above idea to formally derive a second-order approximation for \eq{rect}. See \cref{sec:formal} for the derivation for the general case. Taking into account the next-order term in the approximate Hoeffding decomposition, we have
\begin{equation}\label{rect4}
\begin{split}
p_{\beta_1, \beta_2}(G)\propto&  \exp\left\{\frac{8\beta_2}{n^2}\Big[\square-n p^2\big(V-2np E \big)-n^2 p^3 E \Big] \right\} \\
&\times \exp\left\{ \frac{8\beta_2 p^2}{n} \big[V-2np E \big]  
\right\} p^E (1-p)^{n-E},
\end{split}
\end{equation}
where $V$ denotes the number of two-stars (subgraphs with three vertices and two edges connecting them) in $G$.
Because two-term decomposition further reduces variance, it is natural to use the two-star ERGM
\begin{equation}\label{p2rect}
p_2(G):=p_{2,\beta_1,\beta_2}(G)\propto \exp\left\{ \frac{8\beta_2 p^2}{n} \big[V-2n p E \big]  
\right\} p^E (1-p)^{n-E}
\end{equation}
as a second-order approximation to \eq{rect}.
Two-star ERGMs are highly tractable; see \cite{mukherjee2023statistics} and the earlier work by \cite{park2004statistical}.

Note that Hoeffding decomposition is done under the Erd\"os--R\'enyi model; it was initially unclear whether the remaining term of the decomposition in \cref{rect4} still has a smaller-order variance under the two-star ERGM \eq{p2rect}. Fortunately, we can justify \eq{p2rect} indeed achieves second-order accuracy as explained next. See more details in \cref{sec:result}.

\subsection{Justification and discussion}

We identify a simple graph on $n$ labeled vertices $\{1,\dots, n\}$ with an element $x=(x_{ij})_{1\leq i<j\leq n}\in \{0,1\}^{\mcl{I}}$, where $\mcl{I}:=\{(i,j): 1\leq i<j\leq n\}$ and $x_{ij}=1$ if and only if there is an edge between vertices $i$ and $j$. In this way, the ERGM \eq{rect} induces a random element $X\in \{0,1\}^{\mcl{I}}$. Similarly, \eq{p1rect} and \eq{p2rect} induce a random element $Z$ (first-order approximation) and $Y$ (second-order approximation), respectively, in $\{0,1\}^{\mcl{I}}$. Let $h: \{0,1\}^{\mcl{I}}\to \IR$ be a test function. For $x\in \{0,1\}^{\mcl{I}}$ and $s\in {\mcl{I}}$, define $x^{(s,1)}$ to have 1 in the $s$th coordinate and otherwise the same as $x$, and define $x^{(s,0)}$
similarly except there is a 0 in the $s$th coordinate. Define
\begin{equation}\label{deltah}
\Delta_s h(x):=h(x^{(s,1)})-h(x^{(s,0)}), \quad \|\Delta h\|:=\sup_{x\in \{0,1\}^{\mcl{I}}, s\in {\mcl{I}}}|\Delta_s h(x)|.
\end{equation}
Applying the general result by \cite[Theorem 1.13]{reinert2019approximating} to \eq{rect} in the subcritical region, we have
\begin{equation}\label{rect5}
|\E h(X)-\E h(Z)|\leq C \|\Delta h\|n^{3/2},
\end{equation}
where the constant $C$ depends only on $\beta_1, \beta_2$. The upper bound in \eq{rect5} is, in fact, sharp as will be shown in \cref{prop:sharp}.
Applying the general result in \cref{sec:result}, we have, in a subset of the subcritical region,
\begin{equation}\label{rect6}
|\E h(X)-\E h(Y)|\leq C \|\Delta h\|n.
\end{equation}
Comparing \eq{rect5} with \eq{rect6}, we see that the second-order approximation \eq{p2rect} indeed achieves second-order accuracy.  

The proof of \eq{rect6} relies on recent advances in understanding ERGMs. According to \cite{mukherjee2023statistics}, it appears that the two-star ERGM only changes marginal edge probabilities from $p$ to $p+O(1/n)$ (cf. \cref{eq:MX23}). The $1/\sqrt{n}$ order effect is moving edges to form two-stars (cf. the proof of \cref{prop:sharp}). Therefore, the Hoeffding decomposition under $G(n,p)$ is still suitable for our purpose. We then extend the method of \cite{reinert2019approximating} and \cite{bresler2019stein} to second-order approximations. While the first-order bound needed in \cite{reinert2019approximating} is easier to obtain using the independence structure of $G(n,p)$, our second-order bound requires working under $p_2$ such as in \eq{p2rect}. We make use of the results of \cite{ganguly2019sub} and \cite{sambale2020logarithmic} to obtain the desired bound. Due to our dependence on the works of \cite{mukherjee2023statistics} and \cite{sambale2020logarithmic}, we restrict ourselves to triangle-free ERGMs and a subset of the subcritical region in \cref{thm:main} in \cref{sec:result}.
While the extension to the triangle case may be merely a technical problem, we feel that the naive extension of the above expansion to third or higher-order approximations may not be valid. This is because changing marginal edge probabilities by an amount of order $1/n$ already incurs too much error for the third-order expansion.

As a related problem, according to \cite[Theorem~15]{bresler2018optimal}, there is a statistical test based on one sample of a random graph to distinguish ERGM and the Erd\"os--R\'enyi model. Based on a naive consideration comparing variances, it seems that no test can distinguish ERGMs with their second-order approximations with high probability using a constant number of i.i.d.\ samples. We leave a more careful study in this direction to future work.
 
While this work may have potential applications in network analysis, we do not pursue them within the scope of this paper. For instance, in the field of information dissemination, owing to distinct purposes and environmental conditions, we may aim to encourage specific shapes and scales within the fundamental grid units of information dissemination. The results in this paper suggest that forming two-stars and triangles could be the initial step in achieving these goals.

\subsection{Organization}

The rest of the paper is organized as follows. In \cref{sec:formal}, we derive the formal second-order approximation to general ERGMs in the subcritical region. In \cref{sec:result}, we rigorously justify the approximation in a subset of the subcritical region for triangle-free ERGMs.

\section{Formal expansion}\label{sec:formal}

In this section, we derive the formal second-order approximation to general ERGMs in the subcritical region.

Consider the ERGM \eq{ERGM} in the space $\mcl G_n$ of all simple graphs on $n$ labeled vertices. Suppose $H_1$ is an edge, $\beta_2,\dots, \beta_k$ are all positive parameters.
Define (cf. \eq{varph})
\begin{equation}\label{varph1}
\Phi(a):=\sum_{i=1}^k \beta_i |E(H_i)| a^{|E(H_i)|-1},\quad \varphi(a):=\frac{e^{2\Phi(a)}}{e^{2\Phi(a)}+1},
\end{equation}
where $|E(H_i)|$ is the number of edges in the subgraph $H_l$.
The so-called subcritical region (cf. \cite{bhamidi2011mixing} and \cite{chatterjee2013estimating}) contains all the parameters $\beta=(\beta_1,\dots, \beta_k)$ such that there is a unique solution $p:=p(\beta)$ to the equation $\varphi(a)=a$ in $(0,1)$ and $\varphi'(p)<1$. We always use $p$ to denote the unique solution in the rest of the paper. It can be verified that $p$ satisfies
\begin{equation}\label{eq:p}
2\Phi(p)=\log(\frac{p}{1-p}).
\end{equation}

To formally derive the second-order approximation to the ERGM \eq{ERGM}, we will apply the  Hoeffding decomposition to subgraph counting functions of the Erd\"os--R\'enyi graph $G(n,p)$.
We refer to \cite[Example~2 (cont.) and Lemma 1 on page 344]{janson1991asymptotic} for details. 
Let $\{Z_{ij}: 1\leq i<j\leq n\}$ be i.i.d.\ Bernoulli($p$) random variables representing the presence of $N:={n\choose 2}$ edges. Define $Z_{ji}:=Z_{ij}$ for $1\leq i<j\leq n$.
For $m$ distinct indices $s_1,\dots, s_m\in \mcl{I}:= \{(i,j): 1\leq i<j\leq n\}$, we have the following orthogonal decomposition (sum of uncorrelated mean-zero terms)
\begin{equation}\label{eq:hoeff1}
Z_{s_1}\dots Z_{s_m}-\E Z_{s_1}\dots Z_{s_m}=\sum_{l=1}^m p^{m-l}\sum_{1\leq i_1<\dots i_l\leq m}(Z_{s_{i_1}}-p)\dots (Z_{s_{i_l}}-p).
\end{equation}
In fact, the equation can be checked by the binomial theorem and the uncorrelatedness of summands follows from the independence of the $Z$'s.

Let $E, V$ and $\triangle$ denote the number of edges, two-stars and triangles, respectively, in the graph $G$. If $G$ is random, then $E, V$ and $\triangle$ are regarded as random variables. We subtract the leading terms in their (approximate) Hoeffding decomposition under $G(n, p)$ and define
\begin{equation}\label{eq:wtE}
\wt E:=E-\E_p E=\sum_{1\leq i<j\leq n}(Z_{ij}-p),
\end{equation}
\begin{equation}\label{eq:wtV}
\wt V:=V-\E_p V-2np \wt E\approx \sum_{1\leq i<j<k\leq n}\big[(Z_{ij}-p)(Z_{ik}-p)+(Z_{ji}-p)(Z_{jk}-p)+(Z_{ki}-p)(Z_{kj}-p)\big],
\end{equation}
\begin{equation}\label{eq:wtT}
\wt \triangle:=\triangle-\E_p \triangle -p\wt V-n p^2 \wt E\approx \sum_{1\leq i<j<k\leq n}(Z_{ij}-p)(Z_{ik}-p)(Z_{jk}-p),
\end{equation}
where $\E_p$ denotes the expectation with respect to the Erd\"os--R\'enyi random graph $G(n,p)$.
Since we are interested in large $n$ behavior, for simplicity of notation, we again used an approximate Hoeffding decomposition by neglecting some smaller order terms (namely $n-l\approx n$ for bounded constants $l$).
From \cref{eq:wtE,eq:wtV,eq:wtT}, the number of homomorphisms of an edge to $G$ is
\be{
2E=\E_p (2E)+2\wt E;
}
the number of homomorphisms of a two-star to $G$ is
\be{
2V=\E_p (2V)+2\wt V+4np \wt E;
}
the number of homomorphisms of a triangle to $G$ is
\be{
6 \triangle =\E_p (6\triangle)+6\wt \triangle+6p\wt V +6np^2\wt E.
}
For a general subgraph $H_i$ with $v_i\geq 2$ vertices, $e_i\geq 1$ edges, $s_i\geq 0$ two-stars and $t_i\geq 0$ triangles, 
the leading terms (involving $\wt \triangle, \wt V, \wt E$) in the approximate Hoeffding decomposition of $|\text{hom}(H,G)|$ are
\ben{\label{eq:leading}
\E_p |\text{hom}(H_i,G)|+n^{v_i-3} 6t_i p^{e_i-3}\wt \triangle+n^{v_i-3} 2s_i p^{e_i-2}\wt V +n^{v_i-2} 2e_i p^{e_i-1}\wt E.
}
In fact, from \cref{eq:hoeff1}, it can be checked by combinatorics that the full Hoeffding decomposition for $|\text{hom}(H_i, G)|$, where $H_i$ is a graph with $v_i$ vertices and $e_i$ edges, is
\begin{equation}\label{eq:fulldecomp}
|\text{hom}(H_i,G)|=\E_p|\text{hom}(H_i, G)|+\sum_{G_{ij}\subset H_i} (n-v_{ij})\dots(n-v_i+1) p^{e_i-e_{ij}}  |\text{hom}(G_{ij}, H_i)| \wt{\#(G_{ij}, G)},
\end{equation}
where the sum is over all distinct, nonempty, simple subgraphs $G_{ij}$ of $H_i$ without isolated vertices, $v_{ij}$ and $e_{ij}$ are the number of vertices and edges of $G_{ij}$, respectively, $\wt{\#(G_{ij}, G)}$ denotes the number of (not necessarily induced) copies of $G_{ij}$ in $G$ except that each edge indicator $Z_s$, $s\in \mcl{I}$, is replaced by  $(Z_s-p)$ (cf. the right-hand sides of \cref{eq:wtE,eq:wtV,eq:wtT} for the cases of edges, two-stars and triangles).

From \eq{eq:leading}, the ERGM \eq{ERGM} can be rewritten as
\begin{equation}\label{eq:expansion}
\begin{split}
p_\beta(G)=&\frac{1}{Z(\beta)}\exp\left\{ n^2 \sum_{i=1}^k \beta_i t(H_i, G)\right\}\\
=&\frac{1}{Z(\beta)}\exp\Big\{ Const.+Remain.\Big\}\\
&\times \exp\Big\{ \big(\sum_{i=1}^k \beta_i t_i p^{e_i-3} \big)\frac{6\wt \triangle}{n}+\big(\sum_{i=1}^k \beta_i s_i p^{e_i-2} \big)\frac{2\wt V}{n}\Big\}\exp\Big\{\big(\sum_{i=1}^k \beta_i e_i p^{e_i-1} \big)2\wt E\Big\},
\end{split}
\end{equation}
where $Const.$ is a constant not depending on $G$, $Remain.$ are the remaining terms after subtracting the three leading terms in the approximate Hoeffding decomposition. Because $p$ satisfies \eq{eq:p}, the last factor in \eq{eq:expansion} is proportional to 
$p^E (1-p)^{n-E}$.
Therefore, it is not hard to believe that $G(n,p)$ whose probability distribution is
\ben{\label{eq:p1}
p_1(G):=p_{1,\beta}(G)\propto p^E (1-p)^{n-E}
}
for $p$ satisfying \eq{eq:p} provides a first-order approximation to the ERGM \eq{ERGM}. It is known that the next-order terms in the Hoeffding decomposition are represented by subgraphs with three vertices; this can be checked by direct computation ; see also \cite[Lemma~2]{janson1991asymptotic} for a more general result allowing additional random variables to be associated with vertices. Therefore, based on \cref{eq:expansion}, it is natural to propose the following second-order approximation to ERGM.

\begin{definition}\label{def:second}
For the ERGM \eq{ERGM} in the subcritical region and $p$ satifying \eq{eq:p}, the second-order approximation is given by a random graph whose probability distribution is
\begin{equation}\label{eq:p2}
p_2(G):=p_{2,\beta}(G)\propto \exp\Big\{ \big(\sum_{i=1}^k \beta_i t_i p^{e_i-3} \big)\frac{6\wt \triangle}{n}+\big(\sum_{i=1}^k \beta_i s_i p^{e_i-2} \big)\frac{2\wt V}{n}\Big\}p^E (1-p)^{n-E}
\end{equation}
for $G\in \mcl{G}_n$.
\end{definition}

\section{Justification of the formal expansion}\label{sec:result}

Recall that we identify a simple graph on $n$ labeled vertices $\{1,\dots, n\}$ with an element $x=(x_{ij})_{1\leq i<j\leq n}\in \{0,1\}^{\mcl{I}}$, where ${\mcl{I}}:={\mcl{I}_n}:=\{(i,j):1\leq i<j\leq n\}$ and $x_{ij}=1$ if and only if there is an edge between vertices $i$ and $j$. In this way, the ERGM \eq{ERGM} induces a random element $X\in \{0,1\}^{\mcl{I}}$. Similarly, \eq{eq:p1} and \eq{eq:p2} induces a random element $Z$ (first-order approximation) and $Y$ (second-order approximation), respectively, in $\{0,1\}^{\mcl{I}}$. Let $h: \{0,1\}^{\mcl{I}}\to \IR$ be a test function. 
Recall the definition of $\|
\Delta h\|$ in \eq{deltah}.

Using Stein's method, \cite{reinert2019approximating} justified the first-order approximation of the ERGM \eq{ERGM} by the Erd\"os--R\'enyi graph $p_1$ in \eq{eq:p1} in the subcritical region by proving \begin{equation}\label{eq:p1bound}
|\E h(X)-\E h(Z)|\leq C \|\Delta h\|n^{3/2},
\end{equation}
where $C:=C(\beta, H)$ is a constant depending only on the parameters $\beta_1,\dots, \beta_k$ and subgraphs $H_1,\dots, H_k$ in the definition of the ERGM. The bound \eq{eq:p1bound} turns out to be sharp as shown in the next proposition.
\begin{proposition}\label{prop:sharp}
When $X$ is the rectangle ERGM \eq{rect} in the subcritical region and $Z$ is the Erd\"os--R\'enyi graph $G(n,p)$ with $p$ satisfying \cref{rect2}, there exist some parameter values $\beta_1, \beta_2$ satisfying $\Phi'(1)<2$, where $\Phi(a)$ was defined in \cref{varph}, and a sequence of functions $h_n: \{0,1\}^{\mcl{I}_n}\to \IR$ such that 
\begin{equation*}
    \liminf_{n\to \infty} \frac{|\E h_n(X)-\E h_n(Z)|}{\|\Delta h_n\|n^{3/2}}>0.
\end{equation*}
\end{proposition}
We prove \cref{prop:sharp} at the end of this section. The next theorem is our main result, which shows that the proposed second-order approximation in \cref{def:second} indeed achieves second-order accuracy for triangle-free ERGMs in a subset of the subcritical region (where $\beta_2,\dots, \beta_k$ are sufficiently small positive constants). 

\begin{theorem}\label{thm:main}
For the ERGM \cref{ERGM},
assume $\Phi'(1)<2$ (known as the Dobrushin's uniqueness region), where $\Phi(a)$ was defined in \cref{varph1}. Assume in addition that $s_i=0$ for all $2\leq i\leq k$ (all the $H_i$'s are triangle-free). Then, we have
\begin{equation}\label{eq:p2bound}
|\E h(X)-\E h(Y)|\leq C \|\Delta h\|n,
\end{equation}
where $C:=C(\beta, H)$ is a constant depending only on the parameters $\beta_1,\dots, \beta_k$ and the subgraphs $H_1,\dots, H_k$ in the definition of the ERGM.
\end{theorem}

\begin{proof}[Proof of \cref{thm:main}]

In this proof, let $C:=C(\beta, H)$ denote positive constants depending only on the parameters $\beta_1,\dots, \beta_k$ and the subgraphs $H_1,\dots, H_k$ in the definition of the ERGM and may differ from line to line.

Recall from \cref{varph1} that 
\begin{equation*}
\Phi(a)=\sum_{i=1}^k \beta_i e_i a^{e_i-1},\quad e_i=|E(H_i)|.      
\end{equation*}
We have
\begin{equation*}
\Phi'(a)=\sum_{i=2}^k \beta_i e_i (e_i-1) a^{e_i-2},\quad \Phi'(1)=\sum_{i=2}^k \beta_i e_i (e_i-1).
\end{equation*}
Let $\Phi_2(a)$ be the corresponding function to the second-order approximation \cref{eq:p2} with $t_i=0$ for all $i$ because of the triangle-free condition. Then,
\begin{equation*}
\Phi_2'(1)=\sum_{i=2}^k \beta_i 2s_i p^{e_i-2}.
\end{equation*}
From $e_i(e_i-1)\geq 2s_i$ and the condition that $\Phi'(1)<2$, we have 
\begin{equation}\label{eq:forSS20}
    \Phi_2'(1)<2.
\end{equation}
It can be checked that the condition $\Phi'(1)<2$ is stronger than the subcritical condition; see \cite[Remark~2.3]{ganguly2019sub}. Therefore, both the ERGM \cref{ERGM} and \cref{eq:p2} are in the subcritical region under the condition of \cref{thm:main}.

{\bf Step 1. The general bound.}

Recall the one-to-one correspondence between a graph $G$ and its edge representation vector $x$. Recall also that $X$ and $Y$ are the edge representation of the random graphs \cref{ERGM} and \cref{eq:p2}, respectively.
Straightforward modification of \cite[Theorems~1.1]{reinert2019approximating} and using their Lemma~3.2 for their value of $\rho$ in 
Dobrushin's uniqueness region
lead to the following bound
\begin{equation*}
|\E h(X)-\E h(Y)|\leq C \|\Delta h\|\sum_{s\in \mcl{I}} \E|\Delta_s L(Y)-\Delta_s \wt L(Y)|,
\end{equation*}
where the sum is over $\mcl{I}=\{(i,j): 1\leq i<j\leq n\}$ and for each $s$, the differencing operator $\Delta_s$ was defined in \cref{deltah} and (recall \cref{ERGM}, the last line of \cref{eq:expansion}, \cref{eq:p2} and $t_i=0$ for all $i$ in the triangle-free case)
\begin{equation*}
    L(x)=n^2\sum_{i=1}^k \beta_i \frac{|\text{hom}(H_i, G)|}{n^{v_i}},
\end{equation*}
\begin{equation*}
    \wt L(x)=\big(\sum_{i=1}^k \beta_i s_i p^{e_i-2} \big)\frac{2\wt V}{n} +\big(\sum_{i=1}^k \beta_i e_i p^{e_i-1} \big) 2\wt E.
\end{equation*}
To prove \cref{thm:main}, it suffices to show, for any edge $s$ connecting two vertices $a$ and $b$,
\begin{equation}\label{eq:aim1}
    \E|\Delta_s(L(Y)-\wt L(Y))|\leq \frac{C}{n}.
\end{equation}
In the following proof, we fix the edge $s$.

{\bf Step 2. Hoeffding decomposition under $p_2$.}

From their expressions above, we have
\begin{equation*}
    L(x)-\wt L(x)=\sum_{i=1}^k \frac{\beta_i}{n}\left\{\frac{|\text{hom}(H_i, G)|}{n^{v_i-3}}-s_i p^{e_i-2} 2\wt V-n e_i p^{e_i-1}2\wt E \right\}=:\sum_{i=1}^k\frac{\beta_i}{n} R_i(x).
\end{equation*}
To prove \cref{eq:aim1}, it suffices to prove, for any $1\leq i\leq k$,
\begin{equation}\label{eq:aim2}
    \E|\Delta_s R_i(Y)|\leq C.
\end{equation}

\noindent\textbf{Case 1.} If $H_i$ is an edge, then $R_i(x)=Const.$, where $Const.$ denote constants which are independent of $x$ and hence will vanish after taking the differencing operator $\Delta_s$.

\noindent\textbf{Case 2.} If $H_i$ is a two-star, then, from \cref{eq:wtV},
\begin{equation*}
    R_i(x)=2V-2\wt V-4np\wt E=Const.
\end{equation*}

\noindent\textbf{Case 3.} Under the triangle-tree condition, the only remaining possibilities for $H_i$ are simple graphs with $v_i\geq 4$. For Case 3, from \cref{eq:fulldecomp}, we have
\begin{align*}
    R_i(x)=&Const.+\sum_{G_{ij}\subset H_i:\atop v_{ij}\geq 4} \frac{(n-v_{ij})\cdots (n-v_i+1) p^{e_i-e_{ij}}|\text{hom}(G_{ij}, H_i)|}{n^{v_i-3}}\wt{\# (G_{ij}, G)}\\
    &+\left[\frac{(n-3)\cdots(n-v_i+1)}{n^{v_i-3}}-1\right]p^{e_i-2}(2s_i)\wt V\\
    &+\left[\frac{(n-2)\cdots(n-v_i+1)}{n^{v_i-3}}-n\right]p^{e_i-1}(2e_i)\wt E\\
    =:&Const.+R_{i1}(x)+R_{i2}(x)+R_{i3}(x),
\end{align*}
where, as in \cref{eq:fulldecomp}, the sum is over all distinct subgraphs $G_{ij}$ of $H_i$ without isolated vertices, $v_{ij}$ and $e_{ij}$ are the number of vertices and edges of $G_{ij}$, respectively. Note that under the triangle-free condition of the theorem, the terms $R_{i2}(x)$ and $R_{i3}(x)$ above correspond to the cases $v_{ij}=3$ and $v_{ij}=2$, respectively.

Because $\Delta_s\wt E=1$ and $|\frac{(n-2)\cdots(n-v_i+1)}{n^{v_i-3}}-n|\leq C$, we have $\E|R_{i3}(Y)|\leq C$.

Because $|\Delta_s \wt V|\leq Cn$ and $|\frac{(n-3)\cdots(n-v_i+1)}{n^{v_i-3}}-1|\leq C/n$, we have $\E|R_{i2}(Y)|\leq C$. In fact, $\E|R_{i2}(Y)|=o(1)$, but we don't need it here.

To prove \cref{eq:aim2}, it remains to prove
\begin{equation}\label{eq:aim3}
    \E|\Delta_s R_{i1}(Y)|\leq C.
\end{equation}
Therefore, we fix $G_{ij}\subset H_i$ with $v_{ij}\geq 4$. From the definitions of $\wt{\#(G_{ij},G)}$ below \cref{eq:fulldecomp} and $\Delta_s$, we have
\begin{equation}\label{eq:deltas}
    \Delta_s \wt{\#(G_{ij},G)}=\wt{\#(G_{ij}^{(s)}, G)},
\end{equation}
which is the number of copies of $G_{ij}$ in $G$ containing the edge $s$, except that each remaining edge indicator $Z_r$ is replaced by $(Z_r-p)$.

As in  \cite{sambale2020logarithmic}, as building blocks of the Hoeffding decomposition under $p_2$, we define
\begin{equation}\label{eq:hoeffp2}
    f_{d,A}=\sum_{I\in \mcl{I}^d} A_I \sum_{P\in \mcl{P}(I)}(-1)^{M(P)} \left\{\prod_{J\in P\atop |J|=1}  (Y_J-\E Y_J)\right\} \prod_{J\in P\atop |J|>1} \left\{ \E \prod_{l\in J} (Y_l-\E Y_l) \right\},
\end{equation}
where $d\geq 1$ is an integer, $\mcl{I}=\{(i,j): 1\leq i<j\leq n\}$, $A$ is a $d$-tensor with vanishing diagonal (i.e., $A_I>0$ only if all the $d$ elements in $I$ are distinct),
\begin{equation*}
    \mcl{P}(I)=\{S\subset 2^I: S\ \text{is a partition of}\ I\},
\end{equation*}
$M(P)$ is the number of subsets with more than one element in the partition $P$ and for a singleton set $J=\{l\}$, $Y_J:=Y_l$. For our purpose, it is enough to consider those values of $d$ bounded by the maximum number of edges of subgraphs $H_1,\dots, H_k$ in the following.

From \cite[Theorem~3.7]{sambale2020logarithmic}\footnote{They assumed that the subgraphs $H_1,\dots, H_k$ in the definition of ERGM are all connected at the beginning of their paper. However, as far as we checked their proof, this requirement is not needed.}, under the condition \cref{eq:forSS20},
\begin{equation}\label{eq:SS20}
   \E (f_{d,A})^2\leq C \|A\|_2^2, 
\end{equation}
where $\|A\|_2$ is the Euclidean norm of the tensor $A$ when viewed as a vector.

Next, we express \cref{eq:deltas} in terms of \cref{eq:hoeffp2}. We need the following facts. From \cite[Eq.(34)]{ganguly2019sub}, for any fixed $m\geq 1$ and distinct edges $l_1,\dots, l_m$, we have, in the subcritical region,
\begin{equation}\label{eq:GN19}
    |\E(Y_{l_1} \vert Y_{l_2},\dots, Y_{l_m})-\E Y_{l_1}|\leq \frac{C}{n}.
\end{equation}
From \cite[Lemma~3.3(d)]{mukherjee2023statistics}, we have, in the subcritical region,
\begin{equation*}
    |\E\bar \phi-t|\leq \sqrt{\E(\bar \phi-t)^2}\leq C/n,
\end{equation*}
where $t=2p-1+O(1/n)$ (see their Lemma~1.2(a)) and $\E \bar\phi=\E(\phi_1+\dots+\phi_n)/n=2\E Y_l-1$ for any $l\in \mcl{I}$ (using their Eqs.~(8) and (10), the symmetry and a connection between the total degree and the total number of edges of a graph). Therefore, for any $l\in \mcl{I}$,
\begin{equation}\label{eq:MX23}
    |\E Y_l-p|\leq \frac{C}{n}.
\end{equation}

Decomposing according to the union $L$ of singleton sets in the partition $P$ in \cref{eq:hoeffp2}, we write
\begin{equation*}
    f_{d, A}=\sum_{I\in \mcl{I}^d} A_I \sum_{L\subset I} C_L\left\{\prod_{l\in L} (Y_l-\E Y_l)\right\},
\end{equation*}
where $C_L=1$ if $L=I$ and, from \eq{eq:GN19}, $C_L=O(1/n)$ if $L\ne I$ (in the case that $L$ is the empty set, the product $\prod_{l\in L}$ is understood to be equal to 1). From \cref{eq:MX23}, we further rewrite $f_{d, A}$ as
\begin{equation}
    f_{d, A}=\sum_{I\in \mcl{I}^d} A_I \sum_{L\subset I} C_L' \left\{\prod_{l\in L} (Y_l-p)\right\},
\end{equation}
where $C_L'=1$ if $L=I$ and $C_L'=O(1/n)$ if $L\ne I$.
From \cref{eq:deltas}, for $d=e_{ij}-1$ (total number of edges of $G_{ij}$ subtracted by one fixed edge $s$) and some $\{A_I: I\in \mcl{I}^d\}$ with uniformly bounded entries,
\begin{align}\label{eq:hoeff3}
    &\Delta_s \wt{\#(G_{ij},G)}=\wt{\#(G_{ij}^{(s)}, G)}\nonumber\\
    =&\sum_{I\in \mcl{I}^d} A_I \prod_{l\in I} (Y_l-p) \nonumber\\
    =&f_{d,A}-\sum_{I\in \mcl{I}^d} A_I \sum_{L\subset I:\atop L\ne I} C_L' \left\{\prod_{l\in L} (Y_l-p)\right\}.
\end{align}

{\bf Step 3. Final bound.}

Using the the expression \cref{eq:hoeff3}, we now turn to proving \cref{eq:aim3} for each $G_{ij}\subset H_i$ with $v_{ij}\geq 4$.

First of all, 
\begin{equation}\label{eq:final1}
    \E \left|\frac{(n-v_{ij})\cdots(v-v_i+1)p^{e_i-e_{ij}}|\text{hom}(G_{ij}, H_i)|}{n^{v_i-3}} f_{d,A} \right|\leq \frac{C n^{v_i-v_{ij}}}{n^{v_i-3}} \sqrt{n^{v_{ij}-2}}\leq C,
\end{equation}
where the first inequality is from \cref{eq:SS20} and the fact that the number of non-zero entries of $\{A_I: I\in \mcl{I}^d\}$ involved is of the order $O(n^{v_{ij}-2})$ (the number of choices of the remaining $v_{ij}-2$ vertices of $G_{ij}$ after fixing the two vertices $a, b$ connected by $s$) and the second inequality is because $v_{ij}\geq 4$.

Next, for the remaining terms in \cref{eq:hoeff3}, repeated using the derivation for \cref{eq:hoeff3} and expressing $\prod_{l\in L}(Y_l-p)$ using $f_{d, A}$ with $d\leq |L|$, we obtain the upper bound (recall $C_L'=O(1/n)$)
\begin{align*}
    & \E\left|\frac{(n-v_{ij})\cdots(v-v_i+1)p^{e_i-e_{ij}}|\text{hom}(G_{ij}, H_i)|}{n^{v_i-3}} \sum_{I\in \mcl{I}^d} A_I \sum_{L\subset I:\atop L\ne I} C_L' \left\{\prod_{l\in L} (Y_l-p)\right\}  \right|\\
    \leq &\frac{C}{n^{v_{ij}-3}}\frac{1}{n}\sum_{0\leq v\leq v_{ij}-2} n^v \sqrt{n^{v_{ij}-2-v}},
\end{align*}
where the sum is over all possible number $v$ of isolated vertices of $G_{ij}$ after removing some edges (but keeping at least one edge to map to $s$), $n^v$ is the order of the number of choices of these $v$ isolated vertices, and $n^{v_{ij}-2-v}$ comes from an argument using \cref{eq:SS20} as for the first inequality in \cref{eq:final1}.
For each $0\leq v\leq v_{ij}-2$, we have
\begin{equation*}
    \frac{C}{n^{v_{ij}-3}}\frac{1}{n} n^v \sqrt{n^{v_{ij}-2-v}}\leq C n^{-v_{ij}/2+v/2+1}\leq C.
\end{equation*}
Therefore, \cref{eq:aim3} is proved and the theorem follows.
\end{proof}

\begin{proof}[Proof of \cref{prop:sharp}]
We identify a graph $G$ on $n$ vertices with a ${n\choose 2}$-dimensional vector $x$ as at the beginning of this section. Similarly, a random graph corresponds to a random vector in the same way.
Let $X$ follow the rectangle ERGM \cref{rect}
with probability
\begin{equation*}
    p_{\beta_1,\beta_2}(G)\propto \exp\{\frac{8\beta_2}{n^2}\square+2\beta_1 E \},
\end{equation*}
$Z$ follow the Erd\"os--R\'enyi model $G(n,p)$ with $p$ in \cref{rect2}, and
$Y$ follow the two-star ERGM \cref{p2rect}
\begin{equation*}
    p_2(G)\propto \exp\{\frac{2\wt \beta_2}{n} V+2\wt \beta_1 E\},
\end{equation*}
where, by a rewriting of \cref{p2rect},
\begin{equation*}
    \wt \beta_2=4\beta_2 p^2,\ \wt \beta_1=-8 \beta_2 p^3+\frac{1}{2} \log (\frac{p}{1-p}).
\end{equation*}
For each $n$, choose the test function $h_n(x)$ to be 
\begin{align}\label{eq:defhn}
    h_n(x):=&\E\sum_{i=1}^n \Big\{\big(\frac{2d_i(G)}{n-1}-2p+\frac{W_i}{\sqrt{n\wt \beta_2/2}} \big)^2 1_{\{(\frac{2d_i(G)}{n-1}-2p+\frac{W_i}{\sqrt{n\wt \beta_2/2}})^2\leq \frac{M}{n}\}}\nonumber \\
    &\qquad\qquad +\frac{M}{n}1_{\{(\frac{2d_i(G)}{n-1}-2p+\frac{W_i}{\sqrt{n\wt \beta_2/2}})^2> \frac{M}{n}\}}  \Big\},
\end{align}
where $d_i(G)$ is the degree of the vertex $i$ in the graph $G$, $W_1,\dots, W_n$ are i.i.d.\ standard normal random variables independent of everything else, the expectation is with respect to the $W$'s, $1_{\{\cdots\}}$ is the indicator function, and $M$ is a sufficiently large constant to be chosen below (for \cref{eq:truncation} and \cref{eq:claim5}). From \cref{thm:main}, for $\beta_1, \beta_2$ satisfying $\Phi'(1)<2$, we have $|\E h_n(X)-\E h_n(Y)|\leq C \|\Delta h_n\|n$. Therefore, to prove the proposition, it suffices to show
\begin{equation}\label{eq:suff1}
\liminf_{n\to \infty} \frac{|\E h_n(Y)-\E h_n(Z)|}{\|\Delta h_n\|n^{3/2}}>0
\end{equation}
for some parameter values $\beta_1, \beta_2$ satisfying $\Phi'(1)<2$.
Because changing an edge only changes $d_i(G)$ by 1 for two out of $n$ vertices and because of the trunction in the definition of $h_n(x)$, we have $\|\Delta h_n\|\leq C/n^{3/2}$. Therefore, to prove \cref{eq:suff1}, it suffices to show
\begin{equation}\label{eq:suff2}
    \liminf_{n\to \infty} |\E h_n(Y)-\E h_n(Z)|>0
\end{equation}
for some parameter values $\beta_1, \beta_2$ satisfying $\Phi'(1)<2$.

We first consider the simpler term $\E h_n(Z)$. Because $Z$ follows $G(n,p)$, the degree $d_i(Z)$ of each vertex $i$ follows the binomial distribution $Bin(n-1,p)$. Recall $W_i\sim N(0,1)$ are independent normal variables independent of everything else. We have
\begin{equation}\label{eq:suff3}
    \limsup_{n\to \infty} \E h_n(Z)\leq \limsup_{n\to \infty} \E \sum_{i=1}^n (\frac{2d_i(Z)}{n-1}-2p+\frac{W_i}{\sqrt{n \wt\beta_2/2}})^2=4p(1-p)+\frac{2}{\wt \beta_2}.
\end{equation}

Next, we consider $\E h_n(Y)$. In fact, \cite[Lemma~2.2(a)]{mukherjee2023statistics} implies the following

\noindent\textbf{Claim.} 
For sufficiently large $M$,
$\liminf_{n\to \infty} \E h_n(Y)$ is sufficiently close to $1/\wt a_1$, where
\begin{equation*}
    \wt a_1=\wt \theta-\wt \theta^2 (1-\wt t^2),\ \wt \theta=\frac{\wt \beta_2}{2},\ \wt t=2p-1.
\end{equation*}

Take, for example, $\beta_2=0.16$ (which satisfies $\Phi'(1)<2$), $\beta_1=-0.08$, $p=0.5$, then \cref{rect2} is satisified. We can compute $1/\wt a_1=13.58696>13.5=4p(1-p)+2/\wt \beta_2$.
Then \cref{eq:suff2} follows from \cref{eq:suff3} and the above claim. Therefore,   the proposition follows.

We are left to prove the above claim.

\noindent\textbf{Proof of Claim.}
As in \cite{mukherjee2023statistics}, define random variables 
\begin{equation*}
    \phi_i:=\frac{2d_i(Y)}{n-1}-1+\frac{W_i}{\sqrt{(n-1)\theta}},\  1\leq i\leq n, \quad  \theta:=\frac{\wt \beta_2 (n-1)}{2n}.
\end{equation*}
From \cite[page~6 in the supplementary material]{mukherjee2023statistics}, $\phi:=(\phi_1,\dots, \phi_n)$ has joint probability density function $f_n(\phi)$ and 
\begin{equation*}
    \frac{(n-1) \lambda_2}{2} \sum_{i=1}^n\left(\phi_i-t\right)^2 \leq-\log f_n(\phi) \leq \frac{(n-1) \lambda_1}{2} \sum_{i=1}^n\left(\phi_i-t\right)^2,
\end{equation*}
where $\lambda_1, \lambda_2$ are two positive constants depending only on $\beta_1, \beta_2$ and $t=2p-1+O(1/n)$ (see \cite[Lemma~1.2(a)]{mukherjee2023statistics}).
From a straightforward modification of the proof of \cite[Lemma~3.3(b)]{mukherjee2023statistics}, we obtain, for a sufficiently large constant $M$,
\begin{equation}\label{eq:truncation}
    \E\sum_{i=1}^n  (\phi_i-t)^2 1_{\{\sum_{i=1}^n (\phi_i-t)^2>M\}} \to 0\ \text{as}\ n\to \infty.
\end{equation} 
In fact, for any $M>0$,
$$
\begin{aligned}
&\E\sum_{i=1}^n (\phi_i-t)^2 1_{\{\sum_{i=1}^n (\phi_i-t)^2>M\}}\\  & =\frac{\int_{\mathbb{R}^n} e^{-f_n(\phi)} \sum_{i=1}^n (\phi_i-t)^21_{\left\{\sum_{i=1}^n\left(\phi_i-t\right)^2> M\right\}} d \phi}{\int_{\mathbb{R}^n} e^{-f_n(\phi)} d \phi} \\
& \leq \frac{\int_{\mathbb{R}^n} e^{-\frac{(n-1) \lambda_2}{2} \sum_{i=1}^n\left(\phi_i-t\right)^2}\sum_{i=1}^n (\phi_i-t)^2 1_{\left\{\sum_{i=1}^n\left(\phi_i-t\right)^2>M\right\}} d \phi}{\int_{\mathbb{R}^n} e^{-\frac{(n-1) \lambda_1}{2} \sum_{i=1}^n\left(\phi_i-t\right)^2} d \phi} \\
& \leq\left(\frac{\lambda_1}{\lambda_2}\right)^{\frac{n}{2}} \frac{1}{(n-1)\lambda_2}\E \chi_n^2 1_{\{\chi_n^2 \geq(n-1) \lambda_2 M\}},
\end{aligned}
$$
where $\chi_n^2$ is a chi-square random variable with $n$ degrees of freedom.
From the moment generating function of $\chi_n^2$ ($\E e^{t \chi_n^2}=1/(1-2t)^{n/2}$ for $0<t<1/2$) and Markov's inequality, we obtain the desired result \cref{eq:truncation} for sufficently large $M$.

By \cite[Lemma~2.2(a)]{mukherjee2023statistics} (their $a_1=\wt a_1+O(1/n)$), we have
\begin{equation*}
    \sqrt{n} \left[ \sum_{i=1}^n (\phi_i-\bar \phi)^2-\frac{1}{\wt a_1} \right]\to N(0, \frac{1}{2\wt a_1^2})\ \text{in distribution},
\end{equation*}
where $\bar \phi:=(\phi_1+\dots+\phi_n)/n$. This implies
\begin{equation}\label{eq:claim1}
    \sqrt{n} \left[ \sum_{i=1}^n (\phi_i-t)^2+n(t-\bar \phi)^2-\frac{1}{\wt a_1} \right]\to N(0, \frac{1}{2\wt a_1^2})\ \text{in distribution}.
\end{equation}
By \cite[Lemma~3.3]{mukherjee2023statistics}, we have
\begin{equation}\label{eq:claim2}
    \E[(\bar \phi-t)]^2\leq \frac{C}{n^2}.
\end{equation}
From \cref{eq:claim1} and \cref{eq:claim2}, we have
\begin{equation}\label{eq:claim3}
\sqrt{n}\left[\sum_{i=1}^n (\phi_i-t)^2-\frac{1}{\wt a_1}    \right]\to N(0, \frac{1}{2\wt a_1^2})\ \text{in distribution}.
\end{equation}
From \cref{eq:truncation} and \cref{eq:claim3}, we have
\begin{equation*}
    \E \sum_{i=1}^n (\phi_i-t)^2\to \frac{1}{\wt a_1}.
\end{equation*}
By symmetry, we have $n\E(\phi_i-t)^2\to 1/\wt a_1$,
which implies (recall $t=2p-1+O(1/n)$)
\begin{equation}\label{eq:claim4}
    n \E\left(\frac{2d_i(Y)}{n-1}-2p+\frac{W_i}{\sqrt{n\wt \beta_2/2}}  \right)^2\to \frac{1}{\wt a_1}.
\end{equation}
Moreover, the tightness of the sequence $\{n(\phi_i-t)^2\}_{n=1}^\infty$ (see \cite[Lemma~3.3(c)]{mukherjee2023statistics}) implies the tightness of the sequence $\{n(\frac{2d_i(Y)}{n-1}-2p+\frac{W_i}{\sqrt{n\wt \beta_2/2}}  )^2\}_{n=1}^\infty$.
Eq.\cref{eq:claim4}, together with the tightness and symmetry, leads to the fact that for any $\varepsilon>0$, there exists $M>0$, such that (recall the definition of $h_n$ in \cref{eq:defhn})
\begin{equation}\label{eq:claim5}
    \liminf_{n\to \infty} \E h_n(Y)\geq \frac{1}{\wt a_1}-\varepsilon,
\end{equation}
proving the claim.
\end{proof}



\section*{Acknowledgements}

Ding W.~Y. was partially supported by the National Social Science Fund of China 20CTQ005.
Fang X. was partially supported by Hong Kong RGC GRF 14305821, 14304822, 14303423 and a CUHK direct grant.

\bibliographystyle{apalike}
\bibliography{reference}

\end{document}